\documentclass[11pt,a4paper,final]{article}
\usepackage{amsmath}%
\usepackage{amstext}%
\usepackage{amssymb}%
\usepackage{showkeys}%
\usepackage{epsfig}%
\usepackage{graphicx}%
\usepackage{xcolor}
\usepackage{multicol}
\usepackage{amsthm}
\usepackage{booktabs}
\usepackage{apacite}
\usepackage{multirow}
\usepackage[top=1in, bottom=1.25in, left=0.8in, right=0.8in]{geometry}

\theoremstyle{prop}
\theoremstyle{proof}

\newtheorem{prop}{Proposition}

\providecommand{\keywords}[1]{
\textbf{Keywords:~~~} VAR process, Gordon Growth Model.
}

\begin{document}
\title{Gordon Growth Model with Vector Autoregressive Process}
\author{Battulga Gankhuu\footnote{Department of Applied Mathematics, National University of Mongolia; E-mail: battulga.g@seas.num.edu.mn; Phone Number: 976--99246036}}
\date{}

\maketitle 

\begin{abstract}
In this study, we introduce a Gordon's dividend discount model, based on Vector Autoregressive Process (VAR). We provide two Propositions, which are related to generic Gordon growth model and Gordon growth model, which is based on the VAR process.
\end{abstract}


\section{Introduction}

Dividend discount models (DDMs), first introduced by \citeA{Williams38}, are a popular tool for stock valuation. After \citeA{Williams38} initial model, \citeA{Gordon56} introduced a more sophisticated one, in what dividends change by a defined (deterministic) growth rate. If we assume that a firm will not default in the future, then the basic idea of all DDMs is that the market price of a stock equals the sum of the stock's next period price and dividend discounted at a risk--adjusted rate, known as a required rate of return, see, e.g., \citeA{Brealey20}. By their very nature, DDM approaches are best applicable to companies paying regular cash dividends. For a DDM with default risk, we refer to \citeA{Battulga22a}. As the outcome of DDMs depends crucially on dividend forecasts, most research in the last few decades has been around the proper estimations of dividend development. To obtain higher order moments of stock price, \citeA{Battulga22a} used non homogeneous Poisson process for dividends. By using the DDM and regime--switching process, \citeA{Battulga22b} obtained pricing and hedging formulas for European options and equity--linked life insurance products. To model required rate of return on stock, \citeA{Battulga23b} applied a three--regime model. The result of the paper reveals that the regime--switching model is good fit for the required rate of return. Reviews of some existing deterministic and stochastic DDMs, which model future dividends can be found in \citeA{dAmico20a} and \citeA{Battulga22a}. 

The rest of the paper is organized as follows: In Section 2, for a DDM, whose dividend is modeled by the popular Gordon growth model, we obtain its theoretical price formulas. Also, we provide a Proposition, which deals with two representations of theoretical price of a stock. Section 3 is dedicated to the Gordon growth model, which is based on the VAR process. In this Section, we obtain provide existence conditions of stock price and its second--order moment. Finally, Section 4 concludes the study.

\section{Gordon Growth Model}

In this paper, we assume that there are $m$ companies and the companies will not default in the future. As mentioned before the basic idea of all DDMs is that the market price of a stock equals the sum of the stock's next period price and dividend discounted at the required rate of return. Therefore, for successive prices of $i$--th company, the following relation holds 
\begin{equation}\label{08101}
P_{i,t}=(1+k_{i,t})P_{i,t-1}-d_{i,t},~~~i=1,\dots,m,~t=1,2,\dots,
\end{equation}
where $k_{i,t}$ is the required rate of return on stock, $P_{i,t}$ is the stock price, and $d_{i,t}$ is the dividend, respectively, at time $t$ of $i$--th company. In vector form, the above equation is written by
\begin{equation}\label{08102}
P_t=(i_m+k_t)\odot P_{t-1}-d_t,~~~t=1,2,\dots,
\end{equation}
where $\odot$ is the Hadamard's element--wise product of two vectors, $i_m=(1,\dots,1)'$ is an $(m\times 1)$ vector, consisting of 1, $k_t:=(k_{1,t},\dots,k_{m,t})'$ is an $(m\times 1)$ vector of the required rate of returns on stocks at time $t$, $P_t:=(P_{1,t},\dots,P_{m,t})'$ is an $(m\times 1)$ price vector at time $t$, and $d_t:=(d_{1,t},\dots,d_{m,t})'$ is an $(m\times 1)$ dividend vector at time $t$ of the companies. On the other hand, according to the Gordon growth model, it is the well--known fact that successive dividends of companies are modeled by
\begin{equation}\label{08103}
d_t=(i_m+g_t)\odot d_{t-1},~~~t=1,2,\dots,
\end{equation}
where $g_t=(g_{1,t},\dots,g_{m,t})'$ is an $(m\times 1)$ vector of dividend growth rates at time $t$ of the companies. In the paper, we assume that for $i=1,\dots,m$ and $t=1,2,\dots,$ $k_{i,t}>-1$ and $g_{i,t}>-1$. It follows from equations \eqref{08102} and \eqref{08103} that for $t=1,2,\dots$, vectors of the required rate of returns and dividends of the companies are obtained by
\begin{equation}\label{08104}
k_t=(P_t+d_t)\oslash P_{t-1}-i_m
\end{equation}
and
\begin{equation}\label{08105}
g_t=d_t\oslash d_{t-1}-i_m,
\end{equation}
respectively, where $\oslash$ is the element--wise division of two vectors, 

By repeating equations \eqref{08102} and \eqref{08103}, for $t=0,1,\dots$ and $r=0,1,\dots,$ we obtain two version of the price vector at time $t+r$ of the companies
\begin{equation}\label{08106}
P_{t+r}=\sum_{q=1}^\infty \bigg(\prod_{j=r+1}^{r+q}\big(i_m+g_{t+j}\big)\oslash\big(i_m+k_{t+j}\big)\bigg)\odot\prod_{j=1}^r\big(i_m+g_{t+j}\big)\odot d_t
\end{equation}
and
\begin{equation}\label{08107}
P_{t+r}=\prod_{j=1}^r\big(i_m+k_{t+j}\big)\odot P_t-\sum_{q=1}^r\bigg(\prod_{j=q+1}^r\big(i_m+k_{t+j}\big)\odot \prod_{j=1}^q\big(i_m+g_{t+j}\big)\bigg)\odot d_t,
\end{equation}
where for $q=1,2,\dots$ and generic ($m\times 1$) vectors $o_1,\dots,o_q$, $\prod_{j=1}^qo_j=o_1\odot\dots\odot o_q$ is an element--wise product of the vectors $o_1,\dots,o_q$ and with convention $\sum_{j=q}^{q-1} o_j=0$ and $\prod_{j=q}^{q-1} o_j=i_m$. Consequently, we have two equivalent equations at forecast origin $t$ for future price vector at time $t+r$. The first one depends on infinite number of the future dividend growth rates and required rate of returns and dividend vector at time $t$, while second one depends on finite number of the future dividend growth rates and required rate of returns and dividend vector and price vector at time $t$.

Also, one may write the above DDM equation in the following form
\begin{equation}\label{08108}
P_{i,t}=\exp\{\tilde{k}_{i,t}\}P_{i,t-1}-d_{i,t},~~~i=1,\dots,m,~t=1,2,\dots,
\end{equation}
where $\tilde{k}_{i,t}:=\ln(1+k_{i,t})$ is a log required rate of return on stock at time $t$ of $i$--th company. Equation \eqref{08101} can be written in vector form
\begin{equation}\label{08109}
P_t=\exp\{\tilde{k}_t\}\odot P_{t-1}-d_t,~~~t=1,2,\dots,
\end{equation}
where $\tilde{k}_t:=(\tilde{k}_{1,t},\dots,\tilde{k}_{m,t})'$ is an $(m\times 1)$ log required rate of return vector on stock at time $t$ and for an $(m\times 1)$ generic vector $o=(o_1,\dots,o_m)'$, the above exponential means $\exp\{o\}:=\big(\exp\{o_1\},\dots,\exp\{o_m\}\big)'$. It follows from equation \eqref{08109} that the log required rate of return vector at time $t$ is represented by
\begin{equation}\label{08110} 
\tilde{k}_t=\ln\big((P_t+d_t)\oslash P_{t-1}\big), ~~~t=1,2,\dots.
\end{equation}
One can write equation \eqref{08103} in the following log form
\begin{equation}\label{08111}
\ln\big(d_t\oslash d_{t-1}\big)=\tilde{d}_t-\tilde{d}_{t-1}=\tilde{g}_t,~~~t=1,2,\dots,
\end{equation}
where $\tilde{d}_t:=(\tilde{d}_{1,t},\dots,\tilde{d}_{m,t})'$ is an $(m\times 1)$ log dividend vector at time $t$ with $\tilde{d}_{i,t}:=\ln(d_{i,t})$ and $\tilde{g}_t:=(\tilde{g}_{1,t},\dots,\tilde{g}_{m,t})'$ is an $(m\times 1)$ log dividend growth rate vector at time $t$ with $\tilde{g}_{i,t}:=\ln(1+g_{m,t})$. Thus, in terms of the log dividend growth rates, log dividend at time $t+r$ is written by  
\begin{equation}\label{08112}
\tilde{d}_{t+r}=\tilde{d}_t+\sum_{j=1}^r \tilde{g}_{t+j},~~~r\geq 0.
\end{equation}
Consequently, thanks to equations \eqref{08106}, \eqref{08107}, \eqref{08110}, and \eqref{08112}, the price vectors at time $t+r$ are given by the following equation
\begin{equation}\label{08113}
P_{t+r}=\sum_{q=1}^\infty \exp\bigg\{\sum_{j=1}^r \tilde{g}_{t+j}+\sum_{j=r+1}^{r+q}\Big(\tilde{g}_{t+j}-\tilde{k}_{t+j}\Big)\bigg\}\odot d_t,
\end{equation}
and
\begin{equation}\label{08114}
P_{t+r}=\exp\bigg\{\sum_{j=1}^r\tilde{k}_{t+j}\bigg\}\odot P_t-\sum_{q=1}^r\exp\bigg\{\sum_{j=q+1}^r\tilde{k}_{t+j}+\sum_{j=1}^q\tilde{g}_{t+j}\bigg\}\odot d_t.
\end{equation}

Again, we have two equivalent equations at forecast origin $t$ for future price vector at time $t+r$. The first one depends on infinite number of the future log dividend growth rates and log required rate of returns and dividend vector at time $t$, while second one depends on finite number of the future log dividend growth rates and log required rate of returns and dividend vector and price vector at time $t$. Then, a question arise that which one is dominant in some sense? The following Proposition may answer the question.

\begin{prop}\label{prop04}
Let $\mathcal{F}$ and $\mathcal{G}$ be $\sigma$--field such that $\mathcal{F}\subset \mathcal{G}$ and $X$ is a square integrable random variable, defined on a probability space $(\Omega,\mathcal{G},\mathbb{P})$. Then, conditional on both of the $\sigma$--fields $\mathcal{F}$ and $\mathcal{G}$, the mean squared distance between the random variable $X$ and a conditional expectation $\mathbb{E}(X|\mathcal{G})$ is less than or equal to the distance between the random variable $X$ and a conditional expectation $\mathbb{E}(X|\mathcal{F})$. Also, expectations of the conditional expectations $\mathbb{E}(X|\mathcal{G})$ and $\mathbb{E}(X|\mathcal{F})$ are equal to $\mathbb{E}(X)$.
\end{prop}

\begin{proof}
\begin{eqnarray}\label{08115}
\text{Var}\big(X\big|\mathcal{G}\big)&=&\mathbb{E}\big((X-\mathbb{E}(X|\mathcal{G}))^2\big|\mathcal{G}\big)=\mathbb{E}\big((X-\mathbb{E}(X|\mathcal{F}))^2\big|\mathcal{G}\big)\\
&-&2\mathbb{E}\big((X-\mathbb{E}(X|\mathcal{F}))(\mathbb{E}(X|\mathcal{G})-\mathbb{E}(X|\mathcal{F}))\big|\mathcal{G}\big)+\mathbb{E}\big((\mathbb{E}(X|\mathcal{F})-\mathbb{E}(X|\mathcal{G}))^2\big|\mathcal{G}\big)\nonumber
\end{eqnarray}
Since $\mathbb{E}(X|\mathcal{G})-\mathbb{E}(X|\mathcal{F})$ is measurable with respect to the $\sigma$--field $\mathcal{G}_t$, we have
\begin{equation}\label{08116}
\text{Var}\big(X\big|\mathcal{G}\big)=\mathbb{E}\big((X-\mathbb{E}(X|\mathcal{F}))^2\big|\mathcal{G}\big)-(\mathbb{E}(X|\mathcal{G})-\mathbb{E}(X|\mathcal{F}))^2.
\end{equation}
Consequently,
\begin{equation}\label{08117}
\mathbb{E}\big((X-\mathbb{E}(X|\mathcal{G}))^2\big|\mathcal{G}\big)\leq \mathbb{E}\big((X-\mathbb{E}(X|\mathcal{F}))^2\big|\mathcal{G}\big).
\end{equation}
If we take conditional expectation from the above inequality with respect to the $\sigma$--field $\mathcal{F}$, one gets
\begin{equation}\label{08118}
\mathbb{E}\big((X-\mathbb{E}(X|\mathcal{G}))^2\big|\mathcal{F}\big)\leq \mathbb{E}\big((X-\mathbb{E}(X|\mathcal{F}))^2\big|\mathcal{F}\big).
\end{equation}
Thus, conditional on both of the $\sigma$--fields $\mathcal{F}$ and $\mathcal{G}$, the mean squared distance between the random variable $X$ and a conditional expectation $\mathbb{E}(X|\mathcal{G})$ is less than or equal the distance between the random variable $X$ and a conditional expectation $\mathbb{E}(X|\mathcal{F})$. Finally, $\mathbb{E}(X)=\mathbb{E}\big(\mathbb{E}(X|\mathcal{G})\big)=\mathbb{E}\big(\mathbb{E}(X|\mathcal{F})\big)$. That completes the proof of the Proposition.
\end{proof}

The Proposition tells us that the two forecasts $\mathbb{E}(X|\mathcal{G})$ and $\mathbb{E}(X|\mathcal{F})$ have the same mean, but the forecast $\mathbb{E}(X|\mathcal{G})$ is closer to the random variable $X$ than the forecast $\mathbb{E}(X|\mathcal{F})$ in meaning of mean squared distance. As a result, because for public companies, their prices are observed, to price future theoretical price (forecast) one may use equations \eqref{08106} and \eqref{08113}, that is, the best theoretical price at time $t+r$, which starts at the forecast origin $t$ is given by
\begin{eqnarray}\label{08119}
&&\mathbb{E}[P_{t+r}|\mathcal{G}_t]\nonumber\\
&&=\mathbb{E}\bigg[\prod_{j=1}^r\big(i_m+k_{t+j}\big)\bigg|\mathcal{G}_t\bigg]\odot P_t-\sum_{q=1}^r\mathbb{E}\bigg[\bigg(\prod_{j=q+1}^r\big(i_m+k_{t+j}\big)\odot \prod_{j=1}^q\big(i_m+g_{t+j}\big)\bigg)\bigg|\mathcal{G}_t\bigg]\odot d_t\nonumber\\
&&=\mathbb{E}\bigg[\exp\bigg\{\sum_{j=1}^r\tilde{k}_{t+j}\bigg\}\bigg|\mathcal{G}_t\bigg]\odot P_t-\sum_{q=1}^r\mathbb{E}\bigg[\exp\bigg\{\sum_{j=q+1}^r\tilde{k}_{t+j}+\sum_{j=1}^q\tilde{g}_{t+j}\bigg\}\bigg|\mathcal{G}_t\bigg]\odot d_t,
\end{eqnarray}
where $\mathcal{G}_t$ is an information at time $t$, which includes dividend at time $t$ and price at time $t$. However, equation \eqref{08106} and \eqref{08113} still can be used to determine theoretical price at time $t$ of the companies, that is,
\begin{eqnarray}\label{08120}
\mathbb{E}[P_t|\mathcal{F}_t]&=&\sum_{q=1}^\infty \mathbb{E}\bigg[\bigg(\prod_{j=1}^q\big(i_m+g_{t+j}\big)\oslash\big(i_m+k_{t+j}\big)\bigg)\bigg|\mathcal{F}_t\bigg]\odot d_t\nonumber\\
&=&\sum_{q=1}^\infty \mathbb{E}\bigg[\exp\bigg\{\sum_{j=1}^q\Big(\tilde{g}_{t+j}-\tilde{k}_{t+j}\Big)\bigg\}\bigg|\mathcal{F}_t\bigg]\odot d_t,
\end{eqnarray}
where we use the monotone convergence theorem and $\mathcal{F}_t$ is an information at time $t$, which includes dividend at time $t$ and such that $\mathcal{F}_t\subset \mathcal{G}_t$.

\section{Simple VAR$(p)$ process}

We assume that the log required rate of return vector at time $t$, $\tilde{k}_t$ and the log dividend growth rate vector at time $t$, $\tilde{d}_t$ place first $m$ and next $m$ components of the Bayesian MS--VAR($p$) process with order $p$ and dimension $(2m+\ell)$, respectively. Let us denote the dimension of the Bayesian MS--VAR$(p)$ process by $n$, i.e., $n:=2m+\ell$. As the log required rate of returns and the log dividend growth rates depend on macroeconomic variables and firm--specific variables, such as GDP, inflation, key financial ratios of the companies, and so on, the last $\ell$ components of the VAR$(p)$ process $y_t$ correspond to the economic variables that affect the log required rate of returns and the log dividend growth rates of the companies. By applying the Monte--Carlo methods and equations \eqref{08106}, \eqref{08107}, \eqref{08113}, and \eqref{08114}, one obtains theoretical prices and distributions of the companies. Henceforth, we consider the simple VAR process.

The simple VAR($p$) process $y_t$ is given by the following equation
\begin{equation}\label{08121}
y_t=\nu+A_1y_{t-1}+\dots+A_py_{t-p}+\xi_t,
\end{equation}
where $y_t=(y_{1,t},\dots,y_{n,t})'$ is an $(n\times 1)$ vector of endogenous variables, $\nu=(\nu_1,\dots,\nu_n)'$ is a $(n\times 1)$ intercept vector, $\xi_t=(\xi_{1,t},\dots,\xi_{n,t})'$ is an $(n\times 1)$ Gaussian white noise process with a positive definite covariance matrix $\Sigma$, and for $i=1,\dots,p$, $A_i$ are $(n\times n)$ coefficient matrices. We assume that the VAR$(p)$ process is stable, which means the modulus of all eigenvalues of the matrix $A^*$ is less than one, see \citeA{Lutkepohl05}. To obtain a distribution of the VAR$(p)$ process $y_t$, let us write the VAR$(p)$ process in VAR(1) form
\begin{equation}\label{08122}
y_t^*=\nu^*+A^*y_{t-1}^*+\xi_t^*
\end{equation} 
where $y_t^*:=(y_t',\dots,y_{t-p+1}')'$ is an $(n p\times 1)$ process, $\nu^*:=(\nu',0,\dots,0)'$ is an $(n p\times 1)$ intercept vector, $\xi_t^*:=(\xi_t',0,\dots,0)'$ is a $(np \times 1)$ white noise process, and
\begin{equation}\label{08123}
A^*:=\begin{bmatrix}
A_1 & \dots & A_{p-1} & A_p\\
I_{n} & \dots & 0 & 0\\
\vdots & \ddots & \vdots & \vdots\\
0 & \dots & I_{n} & 0
\end{bmatrix}
\end{equation}
is an $(np\times np)$ matrix. By repeating equation \eqref{08122}, one finds that for $j>0$, 
\begin{equation}\label{08124}
y_{t+j}^*=\sum_{q=1}^jA^{j-q}\nu^*+(A^*)^jy_t^*+\sum_{q=1}^jA^{j-q}\xi_{t+q}^*.
\end{equation}
Let us define an $(n\times np)$ matrix $J:=[I_n:0:\dots:0]$. The matrix can be used to extract the VAR$(p)$ process $y_t$ from the VAR(1) process $y_t^*$, i.e., $y_t=Jy_t^*$. Then, since $J'J\nu^*=\nu$ and $J'J\xi_{t+i}^*=\xi_{t+i}$, we have that for $j>0$,
\begin{equation}\label{08125}
y_{t+j}=\sum_{q=1}^j\Phi_{j-q}\nu+J(A^*)^jy_t^*+\sum_{q=1}^j\Phi_{j-q}\xi_{t+q},
\end{equation}
where $\Phi_q=J(A^*)^qJ'$. 

The matrix $\Phi_q$ can be used to the impulse response analysis. From MA$(\infty)$ representation of the VAR($p$) process, it can be shown that
\begin{equation}\label{08126}
\frac{\partial y_{t+q}}{\partial \xi_t'}=\Phi_q.
\end{equation}
The matrix $\Phi_q$ has the following explanation that a $(i,j)$--th element of the matrix $\Phi_q$, identifies the consequences of a one--unit increase in the $j$--th variable's innovation at date $t$ for the value of the $i$--th variable at time $t+q$, holding all other innovations at all dates constant, see \citeA{Hamilton94} and \citeA{Lutkepohl05}. Let $e_i$ be an $(m\times 1)$ unit vector, whose $i$--th element is 1 and others are zero, $J_k:=[I_m:0:0]$ be an $(m\times n)$ matrix, whose first block is $I_m$ and others are zero and $J_g:=[0:I_m:0]$ be an $(m\times n)$ matrix, whose second block is $I_m$ and others are zero, where $I_m$ is an $(m\times m)$ identity matrix. The matrices can be used to extract the log required rate of return vector and the log dividend growth rate vector from the VAR$(p)$ process, i.e., $\tilde{k}_t=J_ky_t$ and $\tilde{g}_t=J_gy_t$. Then, it follows from equations \eqref{08114} and \eqref{08126} that
\begin{eqnarray}\label{08127}
\frac{\partial P_{t+r}}{\partial \xi_t'}=[a_1':\dots:a_m']',
\end{eqnarray}
where for each $i=1,\dots,m$, the $(1\times n)$ vector $a_i$ is given by
\begin{eqnarray}\label{08127}
a_i&=&\exp\bigg\{\sum_{j=1}^re_i'\tilde{k}_{t+j}\bigg\}P_{i,t}\bigg(e_i'J_k\sum_{j=1}^r\Phi_j\bigg)\\
&-&\sum_{q=1}^r\exp\bigg\{\sum_{j=q+1}^re_i'\tilde{k}_{t+j}+\sum_{j=1}^qe_i'\tilde{g}_{t+j}\bigg\}d_{i,t}\bigg(e_i'J_k\sum_{j=q+1}^r\Phi_j+e_i'J_g\sum_{j=1}^q\Phi_j\bigg).\nonumber
\end{eqnarray}
It follows from equation \eqref{08125} that conditional on the information $\mathcal{F}_t$, expectation and covariance of the process $y_{t+j}$ are given by
\begin{equation}\label{08128}
\mathbb{E}\big(y_{t+j}\big|\mathcal{F}_t\big)=\sum_{q=1}^j\Phi_{j-q}\nu+J(A^*)^jy_t^*,~~~j>0
\end{equation}
and
\begin{equation}\label{08129}
\text{Cov}\big(y_{t+j_1},y_{t+j_2}\big|\mathcal{F}_t\big)=\sum_{q=1}^{j_1\wedge j_2}\Phi_{j_1-q}\Sigma\Phi_{j_2-q}',~~~j_1,j_2>0.
\end{equation}
where the information equals $\mathcal{F}_t:=\sigma(\tilde{d}_0,y_0,\dots,y_t)$. Consequently, as $\xi_t\sim\mathcal{N}(0,\Sigma)$, by equation \eqref{08120} theoretical price at time $t$ of $i$--th company equals
\begin{equation}\label{08130}
\mathbb{E}\big(e_i'P_t\big|\mathcal{F}_t\big)=\sum_{q=1}^\infty \exp\bigg\{e_i'\tilde{d}_t+\mathbb{E}\big(e_i'z_q\big|\mathcal{F}_t\big)+\frac{1}{2}\mathcal{D}\big[\text{Var}\big(e_i'z_q\big|\mathcal{F}_t\big)\big]\bigg\},
\end{equation}
where $z_q:=J_{g,k}\sum_{j=1}^qy_{t+j}$ is $(m\times 1)$ vector, $J_{g,k}:=J_g-J_k$ is an $(m\times m)$ difference matrix and for a generic $(m\times m)$ matrix $O$, $\mathcal{D}[O]$ is a $(m\times 1)$ vector, which consists of diagonal elements of the matrix $O$. On the other hand, by monotone convergence theorem, a mixed moment of stock prices at time $t$ of $i_1$--th and $i_2$--th companies is obtained by
\begin{eqnarray}\label{08131}
\mathbb{E}\big(e_{i_1}'P_tP_t'e_{i_2}\big|\mathcal{F}_t\big)&=&\sum_{q_1=1}^\infty \sum_{q_2=1}^\infty\exp\bigg\{e_{i_1}'\tilde{d}_t+e_{i_2}'\tilde{d}_t\\
&+&\mathbb{E}\big(e_{i_1}'z_{q_1}+e_{i_2}'z_{q_2}\big|\mathcal{F}_t\big)+\frac{1}{2}\mathcal{D}\big[\text{Var}\big(e_{i_1}'z_{q_1}+e_{i_2}'z_{q_2}\big|\mathcal{F}_t\big)\big]\bigg\},\nonumber
\end{eqnarray}
In the following Proposition, we give sufficient conditions of convergences of series \eqref{08130} and \eqref{08131}.

\begin{prop}\label{prop05}
Let all eigenvalues of the matrix $A^*$ be different and their modulus are less than one ($y_t$ is stable). Then, we have that
\begin{itemize}
\item[(i)] if following inequality holds, then series \eqref{08130} is convergent
\begin{equation}\label{08132}
e_i'J_{g,k}\mu+e_i'J_{g,k}\bigg(\frac{1}{2}\Gamma(0)+\Gamma\bigg)J_{g,k}'e_i<0
\end{equation}
\item[(ii)] and if following inequality holds, then series \eqref{08131} is convergent
\begin{eqnarray}\label{08133}
&&\max\Big\{e_{i_1}'J_{g,k}\mu+e_{i_1}'J_{g,k}\big(\Gamma(0)+2\Gamma\big)J_{g,k}'e_{i_1},\\
&&e_{i_2}'J_{g,k}\mu+e_{i_2}'J_{g,k}\big(\Gamma(0)+2\Gamma\big)J_{g,k}'e_{i_2}\Big\}<0,\nonumber
\end{eqnarray}
\end{itemize}
where $\mu:=\big[I-A_1-\dots-A_p\big]^{-1}\nu$ is a mean and $\Gamma(0):=\sum_{i=0}^\infty \Phi_i\Sigma\Phi_i'$ is a covariance matrix of the process $y_t$, respectively, and
\begin{equation}\label{08134}
\Gamma:=\sum_{j_1=1}^{\infty}\sum_{j_2=j_1}^{\infty}\Phi_{j_2}\Sigma\Phi_{j_1-1}'.
\end{equation}
\end{prop}

\begin{proof}
($i$) Let us consider $q$--th term of the series \eqref{08130}, namely,
\begin{equation}\label{08135}
\hat{s}_{i,q}:=e_i'\exp\bigg\{\tilde{d}_t+J_{g,k}\sum_{j=1}^q\mathbb{E}\big(y_{t+j}\big|\mathcal{F}_t\big)+\frac{1}{2}\mathcal{D}\bigg[\text{Var}\bigg(J_{g,k}\sum_{j=1}^qy_{t+j}\bigg|\mathcal{F}_t\bigg)\bigg]\bigg\}.
\end{equation}
Then, according to the ratio test of a series, a ratio of successive terms of the series equals
\begin{eqnarray}\label{08136}
\frac{\hat{s}_{i,q+1}}{\hat{s}_{i,q}}&=&\exp\bigg\{e_i'J_{g,k}\mathbb{E}\big(y_{t+q+1}\big|\mathcal{F}_t\big)+\frac{1}{2}e_i'\mathcal{D}\big[\text{Var}\big(J_{g,k}y_{t+q+1}\big|\mathcal{F}_t\big)\big]\\
&+&e_i'\mathcal{D}\bigg[\text{Cov}\bigg(J_{g,k}y_{t+q+1},J_{g,k}\sum_{j=1}^qy_{t+j}\bigg|\mathcal{F}_t\bigg)\bigg]\bigg\}.\nonumber
\end{eqnarray}
By equations \eqref{08128} and \eqref{08129} and the fact that $y_t$ is a stable process, for the first line of the above equation, it holds
\begin{equation}\label{08137}
\lim_{q\to\infty}\mathbb{E}\big(y_{t+q+1}\big|\mathcal{F}_t\big)=\mu,
\end{equation}
and
\begin{equation}\label{08138}
\lim_{q\to\infty}\text{Var}\big(y_{t+q+1}\big|\mathcal{F}_t\big)=\sum_{q=0}^\infty \Phi_q\Sigma\Phi_q'=\Gamma(0).
\end{equation}
For the second line of equation \eqref{08136}, due to equation \eqref{08129}, we have
\begin{equation}\label{08139}
\text{Cov}\bigg(y_{t+q+1},\sum_{j=1}^qy_{t+j}\bigg|\mathcal{F}_t\bigg)=\sum_{j_1=1}^q\sum_{j_2=j_1}^q\Phi_{j_2}\Sigma\Phi_{j_1-1}'.
\end{equation}
Since all eigenvalues are different, the matrix $(A^*)^j$ can be represented by $(A^*)^j=C\Lambda^jC^{-1}$, where the matrix $C$ consists of eigenvectors of the matrix $A^*$ and $\Lambda$ is a diagonal matrix, whose diagonal elements consist of the eigenvalues of the matrix $A^*$. Consequently, since for generic vectors $o_1,o_2\in \mathbb{R}^n$ and matrix $O\in\mathbb{R}^{n\times n}$, $\text{diag}\{o_1\}O\text{diag}\{o_2\}=O\odot(o_1o_2')$ and $\Phi_j=J(A^*)^jJ'$, it holds
\begin{eqnarray}\label{08140}
\sum_{j_1=1}^q\sum_{j_2=j_1}^q\Phi_{j_2}\Sigma\Phi_{j_1-1}'=\sum_{j_1=1}^q\sum_{j_2=j_1}^qJC\Big(C^{-1}J'\Sigma J(C')^{-1}\odot\big(d_1d_2'\big)\Big)C'J',
\end{eqnarray}
where $d_1:=\mathcal{D}\big[\Lambda^{j_2}\big]$ and $d_2:=\mathcal{D}\big[\Lambda^{j_1-1}\big]$. Therefore, because VAR($p$) process is stable, a generic $(\alpha,\beta)$--th element of the matrix $\sum_{j_1=1}^q\sum_{j_2=j_1}^qd_1d_2'$ converges to
\begin{eqnarray}\label{08141}
\lim_{q\to\infty}\bigg[\sum_{j_1=1}^q\sum_{j_2=j_1}^q\lambda_{\alpha}^{j_2}\lambda_{\beta}^{j_1-1}\bigg]=\frac{\lambda_{\alpha}}{(1-\lambda_{\alpha})(1-\lambda_{\alpha}\lambda_{\beta})}.
\end{eqnarray}
Therefore, the matrix given in equation \eqref{08140} has a finite limit and we denote it by $\Gamma$. As a result, by the ratio test of a series, if condition \eqref{08132} holds, the series \eqref{08130} is convergent.

($ii$) According to the Cauchy--Schwarz inequality, we get that
\begin{eqnarray}\label{08142}
\mathbb{E}\big(e_{i_1}'P_tP_t'e_{i_2}\big|\mathcal{F}_t\big)&\leq&\sqrt{\mathbb{E}\big(e_{i_1}'P_tP_t'e_{i_1}\big|\mathcal{F}_t\big)\mathbb{E}\big(e_{i_2}'P_tP_t'e_{i_2}\big|\mathcal{F}_t\big)}\\
&=&\sqrt{\mathbb{E}\big((e_{i_1}'P_t)^2\big|\mathcal{F}_t\big)\mathbb{E}\big((e_{i_2}'P_t)^2\big|\mathcal{F}_t\big)},
\end{eqnarray}
where for $k=1,2$,
\begin{equation}\label{08143}
\mathbb{E}\big((e_{i_k}'P_t)^2\big|\mathcal{F}_t\big):=\sum_{q_k=1}^\infty \exp\bigg\{2\Big(e_{i_k}'\tilde{d}_t+\mathbb{E}\big(e_{i_k}'z_{q_k}\big|\mathcal{F}_t\big)+\mathcal{D}\big[\text{Var}\big(e_{i_k}'z_{q_k}\big|\mathcal{F}_t\big)\big]\Big)\bigg\}.
\end{equation}
As a result, by repeating the above idea, one obtains equation \eqref{08133}. That completes the proof.
\end{proof}

It is worth mentioning that if VAR$(p)$ process, given by equation \eqref{08121} is cointegrated, then by using the Granger representation theorem (see \citeA{Hansen05}), it can be shown that the series \eqref{08130} is diverging, that is, the theoretical stock price of $i$--th company does not exist. Consequently, higher order moments of stock prices also do not exist.

Finally, to use the Gordon growth model, based on the VAR process, one needs parameter estimation. Parameter estimation method can found in \citeA{Hamilton94} and \citeA{Lutkepohl05}. One may be extend the results in this section to Bayesian framework. In this case, one needs Monte Carlo simulation methods. Recent new Monte Carlo simulation method can be found in \citeA{Battulga24g}.

\section{Conclusion}

In this paper, we study Gordon growth model, based on the VAR process, and provide two Propositions. The first one deals with two representations of stock price, which arise in generic Gordon growth model and is related to stock forecast. The other one is dedicated to the existence conditions of stock price and its second order moment for Gordon growth model, based on a simple VAR process. Also, we obtain a result that if one uses the error correction model (cointegrated) for Gordon growth model, the stock price does not exist.

\bibliographystyle{apacite}
\bibliography{References}

\end{document}